\documentclass[10 pt,english]{article}
\usepackage[activeacute]{babel}
\usepackage{amsfonts} 
\usepackage{mathrsfs} 
\usepackage{amssymb, amsmath, amsfonts}
\usepackage{makeidx}
\usepackage{epsfig}
\usepackage{color}
\usepackage[T1]{fontenc}
\usepackage{graphics,graphicx,graphpap}
\usepackage[ansinew]{inputenc}
\usepackage{amsthm}

\newcommand\blfootnote[1]{%
  \begingroup
  \renewcommand\thefootnote{}\footnote{#1}%
  \addtocounter{footnote}{-1}%
  \endgroup
}

\if@twoside \oddsidemargin 5pt \evensidemargin 5pt \marginparwidth 20pt
 \else \oddsidemargin 5pt \evensidemargin 5pt \marginparwidth 20pt \fi

\newcommand{\ZZ}{\mathbb{Z}}

\newcommand{\Aut}{\mathrm{Aut}}
\newcommand{\G}{\Gamma}

\newcommand{\Cay}{\mathop{\rm Cay}}

\newcommand{\la}{\langle}
\newcommand{\ra}{\rangle}

\newcommand{\HTG}{\mathop{\rm HTG}}
\renewcommand{\Re}{\mathcal{R}}
\newcommand{\Bl}{\mathcal{B}}
\newcommand{\Gr}{\mathcal{G}}
\newcommand{\GPr}{\mathop{\rm GPr}}

\newtheorem{theorem}{Theorem}[section]
\newtheorem{proposition}[theorem]{Proposition}
\newtheorem{corollary}[theorem]{Corollary}
\newtheorem{lemma}[theorem]{Lemma}

\theoremstyle{definition}

\newtheorem{construction}[theorem]{Construction}

\begin{document}

\begin{center}
\Large{\textbf{Symmetries of the Honeycomb toroidal graphs}} \\ [+4ex]
Primo\v z \v Sparl{\small$^{a, b, c}$}
\\ [+2ex]
{\it \small 
$^a$University of Ljubljana, Faculty of Education, Ljubljana, Slovenia\\
$^b$University of Primorska, Institute Andrej Maru\v si\v c, Koper, Slovenia\\
$^c$Institute of Mathematics, Physics and Mechanics, Ljubljana, Slovenia}
\end{center}


\blfootnote{
Email address: 
primoz.sparl@pef.uni-lj.si
}


\hrule

\begin{abstract}
{\em Honeycomb toroidal graphs} are a family of cubic graphs determined by a set of three parameters, that have been studied over the last three decades both by mathematicians and computer scientists. They can all be embedded on a torus and coincide with the cubic Cayley graphs of generalized dihedral groups with respect to a set of three reflections. In a recent survey paper B. Alspach gathered most known results on this intriguing family of graphs and suggested a number of research problems regarding them. In this paper we solve two of these problems by determining the full automorphism group of each honeycomb toroidal graph.
\end{abstract}

\hrule

\begin{quotation}
\end{quotation}

\section{Introduction}
\label{sec:Intro}

In this short paper we focus on a certain family of cubic graphs with many interesting properties. They are called {\em honeycomb toroidal graphs}, mainly because they can be embedded on the torus in such a way that the corresponding faces are hexagons. The usual definition of these graphs is purely combinatorial where, somewhat vaguely, the honeycomb toroidal graph $\HTG(m,n,\ell)$ is defined as the graph of order $mn$ having $m$ disjoint ``vertical'' $n$-cycles (with $n$ even) such that two consecutive $n$-cycles are linked together by $n/2$ ``horizontal'' edges, linking every other vertex of the first cycle to every other vertex of the second one, and where the last ``vertcial'' cycle is linked back to the first one according to the parameter $\ell$ (see Section~\ref{sec:HTGdef} for a precise definition). As was shown in~\cite{AlsDea09} these graphs can alternatively also be described as Cayley graphs of generalized dihedral groups with respect to a set of three reflections which for instance implies that these graphs are vertex-transitive (see Section~\ref{sec:prelim} for the definition of some terms not defined in the Introduction). 

It is thus not surprising that this family of graphs has been studied in various papers, both by mathematicians and by computer scientists. The mathematicians are of course always interested in graphs having nice structural properties and a high degree of symmetry. But one of the more important reasons why these graphs are of particular interest to them is that they stand as the last obstacle to a proof that each bipartite Cayley graph of a generalized dihedral group has the property that any two vertices in different partition sets are linked by a Hamilton path of this graph. In view of the fact that it is not even known whether all Cayley graphs of dihedral groups possess a Hamilton cycle, this would be a very remarkable result. The main point of interest for computer scientists in these graphs is that small valency and high degree of symmetry make them desirable models for (computer) networks. 

In a recent survey paper~\cite{Als??} Alspach discusses the above mentioned different viewpoints regarding these graphs, gives an indication of how they came about and why they are interesting to researchers from different fields of science and gathers most of the known results on the topic. Doing so he gives a number of interesting open problems, two of which concern symmetries of these graphs. In~[Research Problem 4]\cite{Als??} he suggests the problem of determining the full automorphism group of each $\HTG(m,n,\ell)$. Should one be able to solve this problem, the answer to~[Research Problem 5]\cite{Als??}, which asks for the classification of all examples with the smallest possible automorphism group (in this case the group acts regularly on the vertex-set of the graph), would of course also be obtained. 
\bigskip

The purpose of this paper is to solve these two problems by proving the following two theorems (see Section~\ref{sec:additional} for the definition of the graphs $\GPr(n)$).

\begin{theorem}
\label{the:main1}
Let $m$ and $n$ be positive integers, where $n \geq 4$ is even, let $0 \leq \ell \leq n/2$ be an integer of the same parity as $m$, let $\G = \HTG(m,n,\ell)$ and let the group $G$ be as in \eqref{eq:thegroupG}. Then $\G$ is not a normal Cayley graph of $G$ if and only if one of the following holds:
\begin{itemize}\itemsep = 0pt
	\item $\G = \HTG(1,6,3)$ and is isomorphic to the $3$-arc-regular complete bipartite graph $K_{3,3}$;
	\item $\G \in \{\HTG(2,4,0), \HTG(2,4,2), \HTG(1,8,3)\}$ and is isomorphic to the $2$-arc-regular cube graph;
	\item $\G = \HTG(1,14,5)$ and is isomorphic to the $4$-arc-regular Heawood graph;	
	\item $\G \in \{\HTG(1,16,5), \HTG(2,8,4)\}$ and is isomorphic to the $2$-arc-regular M\"obius-Kantor graph;
	\item $\G = \HTG(3,6,3)$ and is isomorphic to the $3$-arc-regular Pappus graph;
	\item $mn = 4n'$ for some integer $n' > 2$, either $n = 4$ or $\G \in \{\HTG(1,4n',2n'-1), \HTG(2,2n',2)\}$, and $\G$ is isomorphic to the generalized prism graph $\GPr(n')$, is not arc-transitive and has vertex-stabilizers of order $2^{n'-1}$.
\end{itemize}
\end{theorem}

\begin{theorem}
\label{the:main2}
Let $m$ and $n$ be positive integers, where $n \geq 4$ is even, let $0 \leq \ell \leq n/2$ be an integer of the same parity as $m$ and let $\G = \HTG(m,n,\ell)$. If $\G$ is none of the graphs from Theorem~\ref{the:main1}, then its automorphism group can be determined via the four conditions
\begin{itemize}
\item[(c1)] $\gcd(n,\ell+m) = 2m$ and $2mn \mid (\ell^2+2m\ell-3m^2)$,
\item[(c2)] $\gcd(n,\ell-m) = 2m$ and $2mn \mid (\ell^2-2m\ell-3m^2)$,
\item[(c3)] $\ell \in \{0,n/2\}$,
\item[(c4)] $\gcd(n,\ell+m) = 2m = \gcd(n,\ell-m)$ and $2mn \mid (\ell^2+3m^2)$,
\end{itemize}
where
\begin{itemize}
\itemsep = 0pt
\item $\G$ is $2$-arc-regular if and only if any two (and thus all) of (c1), (c2), (c3) and (c4) hold, which occurs if and only if $\G$ is one of $\HTG(m,2m,m)$ with $m \geq 4$, and $\HTG(m,6m,3m)$ with $m \geq 2$;
\item $\G$ is $1$-arc-regular if and only if (c4) holds, but none of (c1), (c2) and (c3) holds;
\item $\G$ is not arc-transitive with vertex-stabilizers of order $2$ if and only if precisely one of (c1), (c2) and (c3) holds;
\item $\Aut(\G)$ is regular on $\G$ if and only if none of (c1), (c2), (c3) and (c4) holds.
\end{itemize}
\end{theorem}

\section{Preliminaries}
\label{sec:prelim}

Throughout the paper all graphs are assumed to be finite, connected and undirected. Adjacency is denoted by $\sim$ and the edges are usually given as unordered pairs of vertices. 

For an integer $n$ the ring of residue classes modulo $n$ is denoted by $\ZZ_n$. Therefore, all computations involving elements from $\ZZ_n$ are performed modulo $n$.

For an abelian group $A$ the {\em generalized dihedral group} corresponding to $A$ is the group of order $2|A|$ generated by $A$ and an involution $t$ not in $A$ such that $tat = a^{-1}$ for all $a \in A$.

For a group $G$ and an inverse-closed subset $S \subset G \setminus \{1\}$ the {\em Cayley graph} $\Cay(G;S)$ of $G$ with respect to $S$ has vertex-set $G$ and edge-set $\{\{g,gs\} \colon g \in G, s \in S\}$. The graph $\Cay(G;S)$ is a {\em normal} Cayley graph of $G$ if the left regular representation $G_L$ of $G$ is a normal subgroup of the automorphism group $\Aut(\Cay(G;S))$.

A graph $\G$ is {\em vertex-transitive} if the automorphism group $\Aut(\G)$ of $\G$ acts transitively on the vertex-set of $\G$. For $s \geq 1$ an $s$-arc of $\G$ is a sequence of $s+1$ vertices such that any consecutive two are adjacent and any consecutive three are pairwise distinct. A graph $\G$ is {\em $s$-arc-transitive} if $\Aut(\G)$ acts transitively on the set of all $s$-arcs of $\G$. If this action is regular, $\G$ is said to be {\em $s$-arc-regular}. The term $1$-arc-transitive is abbreviated to {\em arc-transitive}.

\section{The HTG graphs}
\label{sec:HTGdef}

We now review the definition of the honeycomb toroidal graphs and their presentation as Cayley graphs of an appropriate generalized dihedral group. We also fix some terminology pertaining to the two viewpoints that we will be using throughout the rest of the paper. The following is simply a restatement of the definition given in~\cite{AlsDea09}.

\begin{construction}
\label{cons}
Let $m$ and $n$ be positive integers, where $n \geq 4$ is even. For each integer $\ell$ with $0 \leq \ell \leq n-1$, where $\ell$ is of the same parity as $m$, the {\em honeycomb toroidal graph} $\HTG(m,n,\ell)$ is the cubic graph with vertex-set $\{\la i,j\ra \colon i \in \ZZ_m, j \in \ZZ_n\}$ and the following adjacencies:
\begin{itemize}
\itemsep = -2pt
	\item $\la i,j \ra \sim \la i, j \pm 1\ra$ for all $i \in \ZZ_m, j \in \ZZ_n$;
	\item $\la i,j \ra \sim \la i+1, j\ra$ for all $i \in \ZZ_m \setminus \{m-1\}, j \in \ZZ_n$ with $i$ and $j$ of different parity;
	\item $\la m-1,j \ra \sim \la 0, j+\ell \ra$ for all $j \in \ZZ_m$ of the same parity as $m$.	
\end{itemize}
\end{construction}

The reader will notice that $\HTG(m,n,\ell) \cong \HTG(m,n,n-\ell)$, and so we loose nothing by assuming $\ell \leq n/2$, which is what we will usually do (in this case the $\HTG$ graph is said to be in {\em normal form}~\cite{Als??}). 
\medskip

In~\cite{AlsDea09} it was shown that each $\HTG(m,n,\ell)$ is isomorphic to a Cayley graph of a generalized dihedral group. In particular, the following result was proved.

\begin{proposition}\cite[Theorem~3.4]{AlsDea09}
Let $m$ and $n$ be positive integers, where $n \geq 4$ is even, and let $0 \leq \ell \leq n-1$ be an integer of the same parity as $m$. Let 
\begin{equation}
\label{eq:thegroupG}
	G = \la t,x,y \mid t^2, x^{n/2}, y^m = x^{(\ell + m)/2}, txt = x^{-1}, tyt = y^{-1}, xy = yx\ra.
\end{equation}
Then the honeycomb toroidal graph $\HTG(m,n,\ell)$ is isomorphic to the Cayley graph $\Cay(G;\{t,tx,ty\})$. 
\end{proposition}

Observe that the group $G$ is a generalized dihedral group, where $\la x,y\ra$ is the index $2$ abelian subgroup (of order $mn/2$). It is easy to verify that the orders of $x$, $y$ and $x^{-1}y$ are
\begin{equation}
\label{eq:orders}
|x| = n/2,\quad |y| =  \frac{mn}{\gcd(n,\ell + m)} \quad \text{and}\quad |x^{-1}y| = \frac{mn}{\gcd(n,\ell-m)}.
\end{equation}
The following detail from the proof of \cite[Theorem~3.4]{AlsDea09} will be important for us. One of the isomorphisms from $\HTG(m,n,\ell)$ to $\Cay(G;\{t,tx,ty\})$ is given by the following correspondence between the vertices of these two graphs. For each $i \in \{0,1,\ldots , m-1\}$ the vertex $\la i,i_n\ra$, where $i_n \in \{0,1,\ldots , n-1\}$ is such that $i \equiv i_n \pmod{n}$, corresponds to $y^i$. For each $i \in \{0,1,\ldots , m-1\}$ and each $j \in \{0,1,\ldots , n/2-1\}$ the vertex $\la i, (i+2j)_n\ra$ then corresponds to $x^jy^i$, while the vertex $\la i, (i+2j+1)_n\ra$ corresponds to $x^jy^it = tx^{-j}y^{-i}$. 

Throughout the paper we will constantly be switching between these two viewpoints of the $\HTG$ graphs. More precisely, whenever we have a $\HTG$ graph $\G$, we think of its vertices as being the pairs $\la i , j\ra$ from Construction~\ref{cons} and at the same time as being the elements of the above group $G$ via the described correspondence. With this in mind we let 
\begin{equation}
\label{eq:edgecolors}
\Re = \{\{g, gt\} \colon g \in G\}, \quad \Bl = \{\{g,gtx\} \colon g \in G\}, \quad \text{and}\quad \Gr = \{\{g,gty\} \colon g \in G\},
\end{equation}
and call the members of $\Re$, $\Bl$ and $\Gr$ the {\em red}, {\em blue} and {\em green} edges, respectively (see Figure~\ref{fig:first} for two examples of how the edges are colored). In this respect we are thus viewing $\HTG(m,n,\ell)$ as what is known in the literature as a Cayley colored graph. Note that each vertex of $\G$ is incident to one edge of each of the three colors. Moreover, for any pair of colors the subgraph consisting of all the edges of these two colors is a disjoint union of cycles (all of which have the same even length, which equals one of the numbers from~\ref{eq:orders}, depending on the two chosen colors) which in the whole graph are linked together by the edges of the third color in a cyclic fashion. 
\begin{figure}[!h]
\begin{center}
	\includegraphics[scale=1]{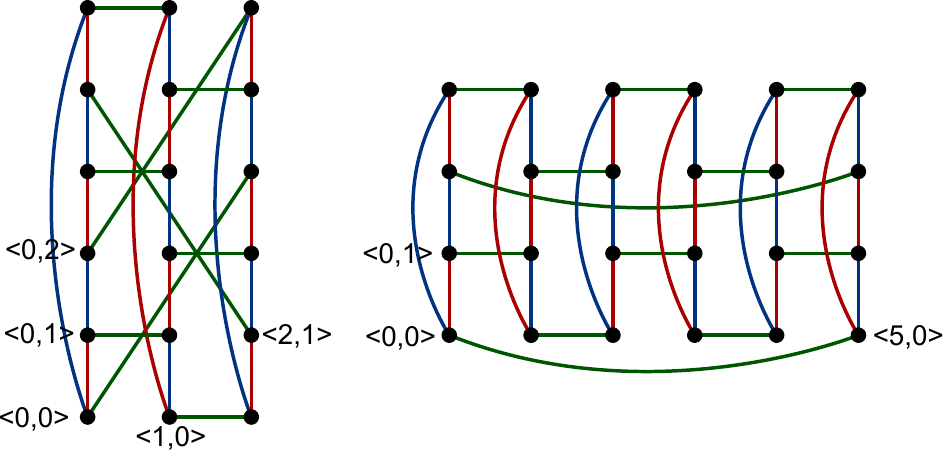}
	\caption{The honeycomb toroidal graphs $\HTG(3,6,3)$ and $\HTG(6,4,0)$.}
	\label{fig:first}
\end{center}
\end{figure}

Let us also point out that, since the index $2$ abelian subgroup $\la x,y \ra$ of $G$ of course acts semiregularly with two orbits on the vertex-set of $\G$, the $\HTG$ graphs are also so-called bi-Cayley graphs on abelian groups. Moreover, as the two orbits of $\la x, y\ra$ are independent sets, they are also so-called Haar graphs of abelian groups. One could thus make use of some results from~\cite{ConZhoFenZha20, EstPis16, FenKovWanYan20, ZhoFen14, ZhoFen16}. Nevertheless, as we will see, in this special setting a direct approach works just fine.

\section{The color permuting automorphisms}
\label{sec:color_perm}

In this section we investigate the automorphisms of the $\HTG$ graphs that permute the colors of the edges as given in~\eqref{eq:edgecolors}. More precisely, let $\G = \HTG(m,n,\ell)$ and let $\Re, \Bl$ and $\Gr$ be as in \eqref{eq:edgecolors}. We investigate the automorphisms of $\G$ which preserve the set $\{\Re, \Bl, \Gr\}$, that is, the automorphisms of $\G$ for which whenever two edges are of the same color, the images of these two edges are also of the same color. Such automorphisms are said to be {\em color permuting} in~\cite{FioFioYeb92, HujKutMorMor16} where this concept was studied in some detail. 

Let $G$ be as in \eqref{eq:thegroupG} and let $G_L$ be the left regular representation of $G$. Of course, each element of $G_L$ is color permuting (in fact, $\Re$, $\Bl$ and $\Gr$ are the three orbits of $G_L$ on the edge-set of $\G$). Denote the set of all automorphisms of $G$ which fix the set $\{t,tx,ty\}$ setwise by $\Aut(G;\{t,tx,ty\})$. Then each element of $\Aut(G;\{t,tx,ty\})$ is a color permuting automorphism of $\G$ which fixes the vertex $1$. As in our case each element of the connection set $\{t,tx,ty\}$ is an involution, it is easy to verify that the converse also holds (but see also~\cite[Lemma~2.1]{FioFioYeb92} or~\cite[Section~5]{HujKutMorMor16}). The set of all color permuting automorphisms of $\G$ that fix the vertex $1$ thus coincides with $\Aut(G;\{t,tx,ty\})$. It now easily follows that the group $\Aut_c(\G)$ of all color permuting automorphisms of $\G$ is the semidirect product $G_L \rtimes \Aut(G;\{t,tx,ty\})$. In the words of~\cite{HujKutMorMor16}, all elements of $\Aut_c(\G)$ are of affine type. It is well known and easy to see (but see for instance the seminal paper by Xu~\cite{Xu98} in which the notion of a normal Cayley graph was introduced) that for a Cayley graph $\G' = \Cay(H;S)$ the group $H_L$ is normal in $\Aut(\G')$ if and only if $\Aut(\G') = H_L \rtimes \Aut(H;S)$. Therefore, in our setting the graph $\G = \HTG(m,n,\ell)$ is a normal Cayley graph of $G$ from~\eqref{eq:thegroupG} if and only if $\Aut(\G) = \Aut_c(\G)$. 
\medskip

In the rest of this section we determine the group $\Aut(G;\{t,tx,ty\})$ (and consequently also $\Aut_c(\G)$). Of course, as the connection set consists of three involutions we have that $\Aut(G;\{t,tx,ty\})$ is isomorphic to a subgroup of the symmetric group $S_3$. We first consider each of the three possibilities of fixing one of the elements of $\{t,tx,ty\}$ and interchanging the other two, and then finally also consider the possibility of permuting all three generators $t$, $tx$ and $ty$ in a cycle of length $3$. Observe that a permutation $\varphi$ of the set $\{t,tx,ty\}$ extends to an automorphism of $G$ if and only if 
\begin{equation}
\label{eq:cond1}
|\varphi(t)\varphi(tx)| = n/2,\quad |\varphi(t)\varphi(ty)| = \frac{mn}{\gcd(n,\ell + m)},\quad |\varphi(tx)\varphi(ty)| = \frac{mn}{\gcd(n,\ell-m)}
\end{equation}
and
\begin{equation}
\label{eq:cond2}
(\varphi(t)\varphi(ty))^m = (\varphi(t)\varphi(tx))^{(\ell+m)/2}.
\end{equation}

\begin{lemma}
\label{le:fixred}
Let the group $G$ be as in \eqref{eq:thegroupG}. Then there exists an automorphism in $\Aut(G;\{t,tx,ty\})$ fixing $t$ and interchanging $tx$ with $ty$ if and only if 
$$
	\gcd(n,\ell+m) = 2m\quad \text{and}\quad 2mn \mid (\ell^2+2m\ell-3m^2).
$$ 
\end{lemma}

\begin{proof}
In this case condition~\eqref{eq:cond1} is equivalent to $n/2 = mn/\gcd(n,\ell+m)$, that is $\gcd(n,\ell+m) = 2m$, while \eqref{eq:cond2} reads $x^m = y^{(\ell+m)/2}$. Assuming the first of these holds we have that $2m$ divides $n$ and $\ell+m = 2qm$ for some $q$ coprime to $n/(2m)$. Then \eqref{eq:thegroupG} implies that \eqref{eq:cond2} is in fact 
$$
	x^m = y^{(\ell+m)/2} = y^{qm} = (y^m)^q = x^{q(\ell+m)/2}, 
$$
and so \eqref{eq:cond2} reads $x^{m} = x^{(\ell+m)^2/(4m)}$, which holds if and only if $n/2$ divides $(\ell^2+2m\ell - 3m^2)/4m$.
\end{proof}

\begin{lemma}
\label{le:fixblue}
Let the group $G$ be as in \eqref{eq:thegroupG}. Then there exists an automorphism in $\Aut(G;\{t,tx,ty\})$ fixing $tx$ and interchanging $t$ with $ty$ if and only if 
$$
	\gcd(n,\ell-m) = 2m\quad \text{and}\quad 2mn \mid (\ell^2-2m\ell-3m^2).
$$
\end{lemma}

\begin{proof}
In this case \eqref{eq:cond1} is equivalent to $\gcd(n,\ell-m) = 2m$, while \eqref{eq:cond2} reads $y^{-m} = (y^{-1}x)^{(\ell+m)/2}$, which is equivalent to $x^{(\ell+m)/2} = y^{(\ell-m)/2}$. If $\gcd(n,\ell-m) = 2m$, then $2m \mid n$ and $\ell - m = 2qm$ for some $q$ coprime to $n/(2m)$. Then $(\ell+m)/2 = (q+1)m$, and so \eqref{eq:thegroupG} implies that \eqref{eq:cond2} is equivalent to
$$
x^{(q+1)m} = y^{qm} = x^{q(\ell+m)/2} = x^{q(q+1)m},
$$
which holds if and only if $1 = x^{(q+1)(q-1)m} = x^{(\ell^2-2m\ell-3m^2)/(4m)}$. The result now follows.
\end{proof}

\begin{lemma}
\label{le:fixgreen}
Let the group $G$ be as in \eqref{eq:thegroupG}. Then there exists an automorphism in $\Aut(G;\{t,tx,ty\})$ fixing $ty$ and interchanging $t$ with $tx$ if and only if $\ell \in \{0,n/2\}$.
\end{lemma}

\begin{proof}
In this case \eqref{eq:cond1} is equivalent to $\gcd(n,\ell+m) = \gcd(n,\ell-m)$, while \eqref{eq:cond2} reads $(x^{-1}y)^m = x^{-(\ell+m)/2}$. By \eqref{eq:thegroupG} the later is equivalent to $x^{(\ell+m)/2} = y^m = x^{(m-\ell)/2}$, that is $x^\ell = 1$. Since this holds if and only if $\ell \in \{0,n/2\}$, we only need to verify that $\gcd(n,n/2+m) = \gcd(n,n/2-m)$ holds. As $(n/2+m) + (n/2-m) = n$, this is clear.
\end{proof}

\begin{lemma}
\label{le:3-cycle_colors}
Let the group $G$ be as in \eqref{eq:thegroupG}. Then there exists an automorphism in $\Aut(G;\{t,tx,ty\})$ of order $3$ if and only if 
$$
	\gcd(n,\ell+m) = \gcd(n,\ell-m) = 2m\quad \text{and}\quad 2mn \mid (\ell^2 + 3m^2).
$$
\end{lemma}

\begin{proof}
Observe that such an automorphism exists if and only if there is an automorphism in $\Aut(G;\{t,tx,ty\})$ mapping $t$ to $tx$, $tx$ to $ty$ and $ty$ back to $t$. In this case \eqref{eq:cond1} is equivalent to $\gcd(n,\ell+m) = \gcd(n,\ell-m) = 2m$. This time \eqref{eq:cond2} reads $x^{-m} = (x^{-1}y)^{(\ell+m)/2}$. Assuming \eqref{eq:cond1} holds we have that $2m$ divides $n$ and $\ell = (2q+1)m$ for some $q$ coprime to $n/(2m)$. Then \eqref{eq:thegroupG} implies that \eqref{eq:cond2} is equivalent to 
$$
	x^{qm} = x^{(\ell-m)/2} = y^{(\ell+m)/2} = y^{(q+1)m} = x^{(q+1)^2 m},
$$
which holds if and only if $1 = x^{(q^2+q+1)m} = x^{(\ell^2+3m^2)/(4m)}$.
\end{proof}

\begin{corollary}
\label{cor:2AT}
Let the group $G$ be as in \eqref{eq:thegroupG}. Then $\Aut(G;\{t,tx,ty\}) \cong S_3$ if and only if $n = 2\ell$ and either $\ell = m$ or $\ell = 3m$.
\end{corollary}

\begin{proof}
Suppose $\Aut(G;\{t,tx,ty\}) \cong S_3$. Then Lemma~\ref{le:fixgreen} implies that $\ell \in \{0,n/2\}$ and Lemma~\ref{le:3-cycle_colors} implies that $\gcd(n,\ell+m) = \gcd(n,\ell-m) = 2m$ and $2mn \mid (\ell^2 + 3m^2)$. It follows that $\ell \neq 0$, and so $\ell = n/2$. Since $2mn$ divides $4(\ell^2 + 3m^2) = n^2 + 12m^2$ and $2m \mid n$, it follows that $2mn \mid 12m^2$, that is $n \mid 6m$. Together with $2m \mid n$ this yields $n \in \{2m, 6m\}$, and so $\ell \in \{m,3m\}$. 

For the converse one simply has to verify that setting $n = 2\ell$ and either $\ell = m$ or $\ell = 3m$ the conditions of Lemma~\ref{le:fixred} and Lemma~\ref{le:fixgreen} are both satisfied, which is clear.
\end{proof}

Under the assumption that Theorem~\ref{the:main1} holds, the results of this section thus prove Theorem~\ref{the:main2} (recall that $\HTG(m,n,\ell)$ is a normal Cayley graph of $G$ from~\eqref{eq:thegroupG} if and only if $\Aut(G;\{t,tx,ty\}) = \Aut(\G)_1$).

\section{Additional automorphisms}
\label{sec:additional}

In view of the results of the previous section it remains to classify those $\G = \HTG(m,n,\ell)$ for which $\Aut(\G) \neq \Aut_c(\G)$ and then determine the automorphism groups of these examples. As we shall see there is only a handful of well known small examples and a very specific infinite family of such graphs. We start with an easy observation (recall the definition of the colors from~\eqref{eq:edgecolors}).

\begin{lemma}
\label{le:mix2AT}
Let $m$ and $n$ be positive integers, where $n \geq 4$ is even, let $0 \leq \ell \leq n-1$ be an integer of the same parity as $m$, and let $\G = \HTG(m,n,\ell)$. If $\alpha \in \Aut(\G)$ is such that there is no edge $e$ such that $e$ and $\alpha(e)$ are of the same color, then $\alpha \in \Aut_c(\G)$.
\end{lemma}

\begin{proof}
The assumption that $\alpha$ preserves the color of no edge of $\G$ and that the three edges incident to a given vertex have different colors implies that for each vertex $v$ and the three edges incident to it, say $e$, $e'$ and $e''$, knowing the color of one of $\alpha(e)$, $\alpha(e')$ and $\alpha(e'')$ completely determines the color of all three of these images. As $\G$ is connected, the result easily follows. 
\end{proof}

The graphs $\HTG(m,n,\ell)$ are bipartite and are of girth at most $6$. In fact, the product of the three elements from $\{t,tx,ty\}$ in any order (where $G$ is as in~\eqref{eq:thegroupG}) is an involution, and so starting at any vertex of the graph and then following any walk that has the property that each three consecutive edges on it are of three different colors results in a $6$-cycle. We call all such $6$-cycles {\em generic}. The set of all generic $6$-cycles is clearly a $G_L$-orbit and there are precisely two generic $6$-cycles through each edge. 

Before stating and proving the next easy but useful result we review the definition of the {\em generalized prism} $\mathrm{GPr}(n)$ from~\cite{EibJajSpa19}, where cubic vertex-transitive graphs of girth at most $5$ were characterized. The graph $\mathrm{GPr}(n)$, where $n \geq 2$, has vertex set $\{(i, j) \colon i \in \ZZ_2, j \in \ZZ_{2n}\}$, each vertex $(i, j)$ is adjacent to $(i, j\pm 1)$ and in addition $(i, j)$ is adjacent to $(i+1, j+1)$ for all even $j$ (see Figure~\ref{fig:SWr} where the isomorphic graphs $\GPr(4)$, $\HTG(4,4,2)$, $\HTG(1,16,7)$ and $\HTG(2,8,2)$ are depicted). These graphs might also be called the {\em split wreath graphs} as $\mathrm{GPr}(n)$ can be obtained from the well-known {\em wreath graph} $W(n) = \Cay(\ZZ_2 \times \ZZ_n ; \{(0,\pm 1), (1,\pm 1)\})$ by performing the ``splitting'' construction with respect to the cycle decomposition consisting of all the ``natural'' $4$-cycles of $W(n)$ (see~\cite[Construction~11]{PotSpiVer13} for details). 
\begin{figure}[!h]
\begin{center}
	\includegraphics[scale=0.75]{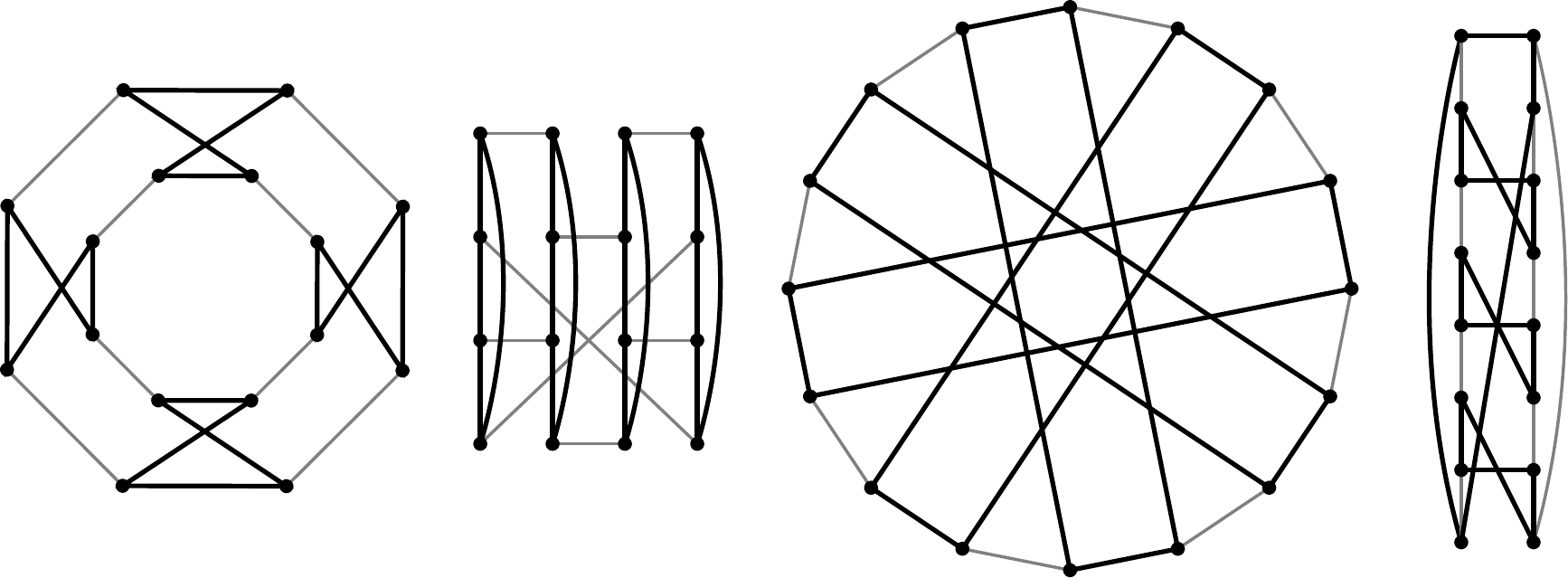}
	\caption{The isomorphic graphs $\GPr(4)$, $\HTG(4,4,2)$, $\HTG(1,16,7)$ and $\HTG(2,8,2)$.}
	\label{fig:SWr}
\end{center}
\end{figure}

\begin{proposition}
\label{pro:fix_one}
Let $m$ and $n$ be positive integers, where $n \geq 4$ is even, let $0 \leq \ell \leq n-1$ be an integer of the same parity as $m$, and let $\G = \HTG(m,n,\ell)$. Let $\Re$, $\Bl$ and $\Gr$ be as in~\eqref{eq:edgecolors}. Then either each automorphism of $\G$ fixing setwise at least one of the sets $\Re$, $\Bl$ and $\Gr$ is color permuting, or $mn$ is divisible by $4$ and $\G \cong \GPr(mn/4)$.
\end{proposition}

\begin{proof}
Suppose $\G$ is not a generalized prism. We show that in this case each $\alpha \in \Aut(\G)$ fixing the set $\Re$ setwise is color permuting (the possibilities that it fixes setwise $\Bl$ or $\Gr$ are dealt with analogously). 

To this end we first show that if $\alpha \in \Aut(\G)$ with $\alpha(\Re) = \Re$ fixes a vertex and all of its neighbors, then $\alpha = 1$. Let $\alpha$ be such an automorphism and let $g$ be a corresponding vertex. Then $\alpha$ fixes the entire blue-green cycle $C$ (the edges are alternatingly blue and green) containing $g$ pointwise. It thus fixes each neighbor of $gtx$ as well as each neighbor of $gty$. If $C$ is of length $mn$ we are done. So suppose this is not the case, let $2s$ be the length of $C$ and consider the blue-green cycle $C'$ containing $gt$. Because of the generic $6$-cycles there is a set of at least $s$ red edges joining the vertices of $C$ to $C'$ (if $mn = 4s$ there are $2s$ of them). As $\G$ is not a generalized prism we have $s \geq 3$, but then the fact that $C$ is fixed pointwise implies that $C'$ must also be fixed pointwise. Therefore, if $\alpha$ fixes $g$ and all of its neighbors, it also fixes the neighbors of each of the three neighbors of $g$. As $\G$ is connected, $\alpha = 1$. 

To complete the proof let $\alpha \in \Aut(\G)$ with $\alpha(\Re) = \Re$. If $\alpha(\Bl) = \Gr$ there is nothing to prove, so assume there is some $e \in \Bl$ with $\alpha(e) \in \Bl$. Let $g , h \in G$ be such that $e = \{g, gtx\}$ and $\alpha(g) = h$. Then $\alpha(gtx) = htx$, and so letting $\beta \in G_L$ correspond to $gh^{-1}$ the automorphism $\beta\alpha$ fixes $g$ and each of its three neighbors. As $\beta\alpha(\Re) = \Re$, the above argument implies that $\alpha = \beta^{-1} \in \Aut_c(\G)$. 
\end{proof}

It thus makes sense to first identify the $\HTG$ graphs (of girth $4$) which are isomorphic to a generalized prism graph. As was pointed out by Alspach~\cite{Als??} the $\HTG$ graphs of girth $4$ are easily identified (recall that we loose nothing by assuming $\ell \leq n/2$). 

\begin{proposition}\cite[Theorem 5.1]{Als??}
\label{pro:girth4}
Let $m$ and $n$ be positive integers, where $n \geq 4$ is even, let $0 \leq \ell \leq n/2$ be an integer of the same parity as $m$, and let $\G = \HTG(m,n,\ell)$. Then $\G$ is of girth $4$ if and only if one of the following holds:
\begin{itemize}\itemsep = 0pt
\item[(i)] $n = 4$;
\item[(ii)] $m = 1$, $n \geq 6$ and $\ell = 3$;
\item[(iii)] $m = 1$, $n \geq 6$, $n \equiv 2 \pmod{4}$ and $\ell = n/2$;
\item[(iv)] $m = 1$, $n \geq 8$, $n \equiv 0 \pmod{4}$ and $\ell = (n-2)/2$;
\item[(v)] $m = 2$, $n \geq 6$ and $\ell \in \{0,2\}$.
\end{itemize}
\end{proposition}

It is not difficult to see that for each $m \geq 2$ and $\ell \in \{0,1,2\}$, where $m$ and $\ell$ are of the same parity, $\HTG(m,4,\ell) \cong \GPr(m)$. Similarly, $\HTG(1,n,(n-2)/2) \cong \GPr(n/4)$ for each $n \geq 8$ with $n \equiv 0 \pmod{4}$. Finally, $\HTG(2,n,2) \cong \GPr(n/2)$ for each even $n \geq 6$. We leave the easy verifications of these claims to the reader (but see Figure~\ref{fig:SWr}). It is also not difficult to see that none of the remaining graphs from Proposition~\ref{pro:girth4} is a generalized prism graph. In fact, for each even $n \geq 6$ the graphs $\HTG(2,n,0)$ and $\HTG(1,2n,3)$ are both isomorphic to the well-known prism graph $\mathrm{Pr}(n) = \Cay(\ZZ_2 \times \ZZ_n ; \{(1,0), (0,\pm1)\})$ of order $2n$, while for each odd $n \geq 3$ the graphs $\HTG(1,2n,n)$ and $\HTG(1,2n,3)$ are both isomorphic to the well-known M\"obius ladder $\mathrm{Ml}(n) = \Cay(\ZZ_{2n} ; \{\pm 1, n\})$ of order $2n$. With the exception of $\mathrm{Ml}(3) \cong K_{3,3}$ (which clearly admits automorphisms that are not color permuting) all of these prisms and M\"obius ladders have edges of two different types regarding the $4$-cycles - one third of them lie on two $4$-cycles each, while the remaining ones lie on one $4$-cycle each. This clearly shows that the vertex stabilizers in the automorphism group of each of these graphs $\G$ are of order $2$, and so Lemmas~\ref{le:fixred}-~\ref{le:fixgreen} imply that for these graphs $\Aut(\G) = \Aut_c(\G)$.
\medskip

For the rest of this section we thus assume that the graphs $\HTG(m,n,\ell)$ we are dealing with are of girth $6$. Suppose then that $\G = \HTG(m,n,\ell)$ is of girth $6$ and that $\alpha \in \Aut(\G)$ is not color permuting. By Proposition~\ref{pro:fix_one} none of the sets $\Re$, $\Bl$ and $\Gr$ is preserved by $\alpha$, and so $\G$ is arc-transitive. Moreover, as there is a set in $\{\Re, \Bl, \Gr\}$, say $\Re$, such that there exist edges $e, e' \in \Re$ with $\alpha(e)$ and $\alpha(e')$ being of different colors, there in fact exists a $3$-path $P$ whose initial and terminal edge $e$ and $e'$ are both red but $\alpha(e)$ and $\alpha(e')$ are of different colors. But then $\alpha(P)$ has all three edges of different colors and thus lies on a generic $6$-cycle, implying that $P$ also lies on a $6$-cycle. Thus, $\G$ has $6$-cycles other than the generic ones, proving that there are at least three $6$-cycles through each edge of $\G$. In fact, since Lemma~\ref{le:mix2AT} clearly implies that multiplying $\alpha$ by an appropriate element of $G_L$ gives an automorphism of $\G$ fixing the vertices $1$ and $t$ but interchanging $tx$ and $ty$, the graph $\G$ is $2$-arc-transitive, and so each $3$-path lies on a $6$-cycle. It is now simply a matter of determining all $\HTG$ graphs with this property. Given the special structure of the $\HTG$ graphs this is fairly easy to do (for instance, considering a red-blue-red $3$-path yields $n = 6$ or $m \leq 2$, and then considering a red-green-red $3$-path yields that $m > 2$ only for $\HTG(3,6,3)$), but instead of describing the argument we simply rely on a more general result from the literature. Namely, it follows from~\cite[Lemma~5.3]{ConZhoFenZha20} or \cite[Theorem~1]{PotVid??} (but see also~\cite{KutMar09}) that the graph $\G$ must be the Heawood graph, the Pappus graph or the M\"obius-Kantor graph (it can easily be verified that the Desargues graph is not a $\HTG$ graph). We are now ready to prove Theorem~\ref{the:main1}.
\medskip

\noindent
{\em Proof of Theorem~\ref{the:main1}:} Recall that $\G = \HTG(m,n,\ell)$ is not a normal Cayley graph of $G$ if and only if $\Aut(\G) \neq \Aut_c(\G)$. By the previous paragraph the only possible examples of girth $6$ are the Heawood graph, the Pappus graph and the M\"obius-Kantor graph. The automorphism groups of these graphs are well known, while the fact that these are indeed precisely the $\HTG$ graphs given in the theorem is easily verified. Moreover, as $\Aut(G;\{t,tx,ty\}) \leq S_3$ the only one of these that could possibly be normal with respect to $G$ is the M\"obius-Kantor graph $\HTG(1,16,5) \cong \HTG(2,8,4)$. However, by Corollary~\ref{cor:2AT} this graph is also not a normal Cayley graph of $G$. 

The discussion from the paragraph following Proposition~\ref{pro:girth4} shows that the only candidates for the girth $4$ $\HTG$ graphs that are not normal Cayley graphs of $G$ are the ones given in the theorem. That $K_{3,3}$ and the cube graph are indeed isomorphic to the $\HTG$ graphs stated in the theorem is easily verified, while their automorphism groups are well known. Again, Corollary~\ref{cor:2AT} shows that the cube graph is not a normal Cayley graph of $G$. Finally, note first that for $m \geq 3$ the edges of $\GPr(m)$ are of two different types. One third of them lies on no $4$-cycle while each of the remaining ones lies on a unique $4$-cycle. The ones that lie on no $4$-cycle thus clearly constitute an orbit of $\Aut(\GPr(m))$, which shows that $\Aut(\GPr(m))$ is isomorphic to the automorphism group of the corresponding wreath graph $W(m)$ (see also~\cite{PotSpiVer13}). The later can be determined easily and can be seen to be generated by an elementary abelian $2$-group of rank $m$ and a dihedral group of order $2m$ (but see for instance~\cite{PraXu89}). Since the vertex stabilizers are of order $2^{m-1}$ and $m \geq 3$, the fact that $\Aut(G;\{t,tx,ty\}) \leq S_3$ implies that this graph is not a normal Cayley graph of $G$.
\hfill $\square$

\section*{Acknowledgements}

\noindent
The author acknowledges support by the Slovenian Research Agency (research program P1-0285 and research projects J1-9108, J1-9110, J1-1694, J1-1695, J1-2451).

\end{document}